\documentclass[12pt]{article}

\usepackage[a4paper]{anysize}\marginsize{2.5cm}{2.5cm}{1cm}{1.8cm}
\pdfpagewidth=\paperwidth \pdfpageheight=\paperheight
\usepackage[latin1]{inputenc}
\usepackage{amssymb}

\usepackage{pgf,tikz,pgfplots}

\usepgfplotslibrary{fillbetween}
\pgfplotsset{every axis/.append style={
axis x line=middle,    % put the x axis in the middle
axis y line=middle,    % put the y axis in the middle
axis line style={<->}, % arrows on the axis
ticks=none
}}

\makeatletter
\renewcommand\section{\@startsection{section}{1}{\z@}%
                                   {-3.5ex \@plus -1ex \@minus -.2ex}%
                                   {2.3ex \@plus.2ex}%
                                   {\normalfont\large\bfseries\centering}}
\renewcommand\subsection{\@startsection{subsection}{2}{\z@}%
                                     {-3.25ex\@plus -1ex \@minus -.2ex}
                                     {1.5ex \@plus .2ex}%
                                     {\normalfont\large\bfseries}}
\renewcommand{\@seccntformat}[1]{\csname the#1\endcsname\ifnum\pdfstrcmp{#1}{section}=0 .\fi\quad}
\renewcommand\@begintheorem[2]{\trivlist\item[\hskip\labelsep{\bfseries#1 #2.}]\it}
\renewcommand\@opargbegintheorem[3]{\trivlist\item[\hskip\labelsep{\bfseries#1 #2}] {\bfseries(#3).}\enspace\it\ignorespaces}
\renewenvironment{abstract}{\begin{quote}\hrulefill\par\footnotesize\textbf{\abstractname.}}{\par\vskip-0.5\baselineskip\hrulefill\end{quote}}
\makeatother

\newtheorem{introtheorem}{Theorem}  % separat numeriert, für die Einleitung
\newtheorem{introcorollary}[introtheorem]{Corollary}  % separat numeriert, für die Einleitung
\newtheorem{introproposition}[introtheorem]{Proposition}  % separat numeriert, für die Einleitung
\newtheorem{introconjecture}[introtheorem]{Conjecture}  % separat numeriert, für die Einleitung

\newtheorem{thm}{Theorem}[section]
\newtheorem{theorem}[thm]{Theorem}
\newtheorem{lemma}[thm]{Lemma}
\newtheorem{proposition}[thm]{Proposition}
\newtheorem{corollary}[thm]{Corollary}
\newtheorem{conjecture}[thm]{Conjecture}

\newcommand\mkthm[2]{\newenvironment{#1}{\begin{#2}\rm}{\end{#2}}}
 \mkthm{definition}{tdefinition}
         \mkthm{remark}{tremark}
       \mkthm{example}{texample}
     \mkthm{question}{tquestion}

\newtheorem{thevarthm}[thm]{\varthmname}

\newenvironment{varthm*}[1]{\trivlist\item[]{\bf #1.}\it}{\endtrivlist}

\newenvironment{proof}[1][Proof]{\trivlist\item[\hskip\labelsep{\textit{#1.}}]}{\hspace*{\fill}$\Box$\endtrivlist}

\let\tilde=\widetilde

\renewcommand\O{\mathcal O}

\newcommand\engqq[1]{``#1''}
\renewcommand\emptyset{\varnothing}  % schönere leere Menge
\renewcommand\ge{\geqslant}  % schräge Variante
\newcommand\keywords[1]{{\renewcommand\thefootnote{}\footnotetext{\emph{Keywords:} #1.}}}
\newcommand\subclass[1]{{\renewcommand\thefootnote{}\footnotetext{\emph{Mathematics Subject Classification (2010):} #1.}}}

\newcommand\Q{\mathbb Q}
\newcommand\R{\mathbb R}
\newcommand\Z{\mathbb Z}
\newcommand\N{\mathbb N}

\newcommand\be[1][@{\;}r@{\;}c@{\;}l@{\;}l@{\;}]{$$\everymath{\displaystyle}\renewcommand\arraystretch{1.2}\begin{array}{#1}}
\newcommand\ee{\end{array}$$}

\newcommand\matr[1]{\left(\begin{array}{*{20}{c}} #1 \end{array}\right)}

\newcommand\tfrac[2]{{\textstyle\frac{#1}{#2}}}  % Bruch im \textstyle
\newcommand\compact{\itemsep=0cm \parskip=0cm}
\newcommand\set[1]{\left\{#1\right\}}
\newcommand\with{\,\,\vrule\,\,}

\newcommand\q[3]{\frac{#1\cdot #2}{\mult_{#3}(#2)}}  % Seshadri-Quotient, z.B. \q LCx

\newenvironment{bycases}{\left\{\begin{array}{@{}l@{\quad}l}}{\end{array}\right.}
\newcommand\eqnref[1]{(\ref{#1})}

\newcommand\newop[2]{\newcommand#1{\mathop{\rm #2}\nolimits}}

\newop\Bl{Bl}
\newop\SD{SD}
\newop\SC{SC}
\newop\rk{rk}
\newop\Vol{Vol}
\newop\Amp{Amp}
\newop\Aut{Aut}
\newop\Nef{Nef}
\newop\Div{Div}
\newop\NS{NS}
\newop\mult{mult}
\newop\End{End}

\newcommand\eps{\varepsilon}

\begin{document}

   \title{Seshadri constants on abelian surfaces}
   \author{\normalsize Maximilian Schmidt}
   \date{\normalsize \today}
   \maketitle
   \thispagestyle{empty}
   \subclass{14C20, 14H52, 14K12, 26A30}

\begin{abstract}
So far, Seshadri constants on abelian surfaces are completely understood only in the cases of Picard number one and on principally polarized abelian surfaces with real multiplication.
Beyond that, there are partial results for products of elliptic curves.
In this paper, we show how to compute the Seshadri constant of any nef line bundle on any abelian surface over the complex numbers. 
We develop an effective algorithm depending only on the basis of the Néron-Severi group to compute not only the Seshadri constants but also the numerical data of their Seshadri curves.
Access to the Seshadri curves allows us to plot Seshadri functions and better understand their structure.
We show that already in the case of Picard number two the complexity of Seshadri functions can vary to a great degree.
Our results indicate that aside from finitely many cases the complexity of the Seshadri function is at least as high as in the Cantor function.
\end{abstract}

%*****************************************************************************

\section*{Introduction}
Recall that for an ample line bundle $L$ on a smooth projective variety $X$, the \emph{Seshadri constant} of $L$ at a point $x\in X$ is by definition the real number
\be
\eps(L,x)=\inf\set{\q LCx\with C\mbox{ irreducible curve through } x}\,.
\ee
On abelian varieties, where this invariant is independent of the chosen point $x$, we write simply $\eps(L)$.
We refer to \cite[Chapt.~5]{Lazarsfeld:PAG}, \cite{Bauer-et-al:primer} and \cite{BL:CAV} for more background on Seshadri constants and abelian surfaces.
So far, Seshadri constants on abelian surfaces are completely understood only in the cases of Picard number one \cite{Bauer:sesh-alg-sf} and on principally polarized abelian surfaces with real multiplication \cite{Bauer-Schmidt}.

The investigation in \cite{Bauer-Schmidt} relies on so-called Pell bounds and on their strong connection with submaximal ample curves.
This tool depends heavily on specificities of abelian surfaces with real multiplication, where the Néron-Severi group is of rank two, the intersection matrices have a simple (known) structure, and there are no elliptic curves to consider.
In the present study, we consider abelian surfaces in general.
We, therefore, need to go beyond \cite{Bauer-Schmidt} in two essential respects:
Firstly, we have to argue without reference to any specific intersection forms, and secondly, we have to take the potential presence of elliptic curves into account.
It turns out that every elliptic curve computes the Seshadri constant of some ample line bundle -- so every elliptic curve has an impact on the shape of the Seshadri function.
To control the influence of elliptic curves, we will develop a method for determining the minimal degree of elliptic curves with respect to a fixed ample line bundle (see Sect.~\ref{sec:sesh-elliptic-case}).
We will show in Sect.~\ref{sec:submaximal-curves-Pell-bounds} how the idea of Pell bounds can be generalized to work with ample and with elliptic curves in this general setting.

By solving the aforementioned problems, we will be able to not only compute Seshadri constants but also Seshadri curves. 
\begin{introtheorem}\label{introthm:seshadri-algorithm}
There is an algorithm that computes the Seshadri constant and all submaximal Seshadri curves of any given nef line bundle on abelian surfaces.
\end{introtheorem}
The methods also reveal that the Seshadri constants can be computed entirely from Pell bounds and elliptic curves, which both can be computed solely from the intersection matrix.
Therefore, Seshadri constants only depend on the numerical data of the intersection matrix.
\begin{introcorollary}\label{introcor:seshadri-numerical}
The Seshadri constants on an abelian surface only depend on the numerical data given by the Néron-Severi group.
More explicitly: Seshadri constants only depend on the intersection matrix.
\end{introcorollary}
Consequently, we can assign \engqq{theoretic} Seshadri functions to any integral quadratic form of dimension $n$ with signature $(1,n-1)$ even if there does not exist an abelian surface with such an intersection matrix.

As an immediate application of the previous Corollary, we see that the global structure of the Seshadri function exhibits a regular behavior that depends on the subgroup of isometries of the Néron-Severi group that leaves the intersection matrix, and hence the Seshadri function, invariant.
This shows that the result from \cite[Thm.~B]{Bauer-Schmidt} holds in fact for all abelian surfaces, regardless of their endomorphism type.

\begin{introcorollary}\label{introcor:isometries}
Let $A$ be an abelian surface and let $G\subset \Aut(\NS(A))$ be the subgroup of isometries with respect to the intersection product that leave the nef cone invariant.
Then the Seshadri function on $A$ is invariant under $G$, i.e., we have $\varepsilon(L)=\varepsilon(\varphi(L))$ for all $\varphi\in G$ and $L\in\Nef(A)$.

Furthermore, $G$ gives rise to a decomposition of the ample cone into subcones such that $G$ acts transitively on the set of subcones and, therefore, the Seshadri function is completely determined by the values of any such subcone.
\end{introcorollary}

We will see that -- on \emph{any} smooth projective surface -- the Seshadri function is locally piecewise linear around line bundles $L$ with $\varepsilon(L,x)<\sqrt{L^2}$.
On the other hand, we know from \cite{Bauer-Schmidt} that its global structure can be quite intricate (with complexity similar to the Cantor function).
We show that piecewise linear Seshadri functions are surprisingly rarer than one might expect.

\begin{introproposition}
Let $A$ be an abelian surface such that the Seshadri function is piecewise linear.
Then $A$ is either a simple abelian surface with Picard number one or is a non-simple abelian surface with Picard number two.
\end{introproposition}
Moreover, we will see that the converse is false since there are non-simple abelian surfaces of Picard number two whose Seshadri functions are not piecewise-linear (see example.~\ref{ex:non-piecewise-linear}).
In fact, we conjecture that there are only finitely many intersection matrices in Picard number two that yield a piecewise linear Seshadri function (see Conj.~\ref{conj:finitely-piecewise}).

The question arises as to what \engqq{simple} structures for the Seshadri function are possible when it is not globally piecewise linear.
One simple structure would emerge if it were to happen that the Seshadri function restricted to any subcone given by the decomposition from Cor.~\ref{introcor:isometries} was linear or at least piecewise linear.
By \cite{Bauer-Schmidt} this cannot occur on principally polarized abelian surfaces with real multiplication since on these surfaces the Seshadri function has infinitely many linear segments on each subcone given by Cor.~\ref{introcor:isometries}.
However, we will show that it does occur on self-products $E\times E$ of elliptic curves $E$ without complex multiplication, as the ample cone on these surfaces has a decomposition such that the restriction of the Seshadri function is linear and each subcone generates the whole Seshadri function via automorphisms of the Néron-Severi group (see Example~\ref{ex:product-elliptic}).
For a discussion of another simple structure of the Seshadri function, see Example~\ref{ex:non-piecewise-linear}.

The algorithm provided for the proof of Thm.~\ref{introthm:seshadri-algorithm} allows us to plot Seshadri functions in Picard number two and further illustrate and study their structure.
In Sect.~\ref{sec:structure-pic-2} we will shed light on the question as to what other structures can occur besides the known cases of piecewise linear and broken linear functions \cite{Bauer-Schmidt}, and how often they occur, based on computational evidence.
In Picard number two, the Seshadri function restricted to a specific compact cross section of the Nef cone yields a real function $\varepsilon:[-1,1]\to \R_{\geq 0}$ (see sect.~\ref{sec:computing-seshc}).
We will show that $\varepsilon$ is locally piecewise linear aside from a nowhere dense subset $Z\subset [-1,1]$.
We may think of the set $Z$ as a measure of the complexity of the Seshadri function.
There are cases where $Z$ is empty and, therefore, the Seshadri function is globally piecewise linear.
However, computational data that we were able to obtain with our algorithm suggest the following:
\begin{introconjecture}
Let $A$ be an abelian surface with Picard number two.
\begin{itemize}\compact
\item[(i)]
If $A$ is simple, then $Z$ is perfect.
\item[(ii)]
If $A$ is non-simple, then $Z$ is either perfect or finite with $|Z|\in \set{0,1,2}$.
Moreover, there are only finitely many intersection matrices such that $Z$ is finite.
\end{itemize}
\end{introconjecture}

Throughout we work over the field of complex numbers.

%*****************************************************************************

\section{Properties of Seshadri functions on projective surfaces}
For general background on Seshadri constants, we refer to \cite[Chapt.~5]{Lazarsfeld:PAG} and \cite{Bauer-et-al:primer}.
First, we introduce some notations.
Let $S$ be a smooth projective surface.
We denote by $\Amp(S)$ and, respectively, $\Nef(S)$ the subcones in the Néron-Severi vector space $\NS_\R(S):=\NS(S)\otimes \R\cong \R^{\rho(S)}$ containing all ample or, respectively, nef classes of $\NS_\R(S)$.
The definition of Seshadri constants extends naturally to classes of the nef cone which yields a continuous and concave function $\varepsilon_x:\Nef(S)\to \R_{\geq 0}$ (see \cite{Nakamaye:base-loci} or \cite[Thm.~6.2.]{ELMNP:moving-seshadri}).
Since Seshadri constants are homogeneous, it is enough to consider the Seshadri function on a compact cross-section of the nef cone.

Let $D$ be an effective divisor on $S$.
Then $D$ defines for any fixed $x\in D$ a linear function $\varphi_{C,x}(L):=L\cdot D/\mult_x D$ on the Néron-Severi vector space.
We call $D$ \textit{submaximal} (for $L$ at $x$), if the value $\varphi_{C,x}(L)$ is strictly smaller than the general upper bound $\sqrt{L^2}$ for the Seshadri constant $\varepsilon(L,x)$ which is valid for every nef line bundle $L$.
An irreducible curve computing $\varepsilon(L,x)$ is called \textit{Seshadri curve} (of $L$ in $x$).

The open sets in $\Nef(S)$ where curves are submaximal or, more generally, where a given linear function is strictly smaller than the upper bound $L\mapsto \sqrt{L^2}$ will play an important role in computing the Seshadri constants on abelian surfaces.w
In the sequel, we will need the following Lemmas.
We omit their elementary proofs.
\begin{lemma}\label{lem:sg-convex}
Let $S$ be a smooth projective surface and $f:\NS_\R(S)\to \R$ a linear function.
Then the set
\be 
\SC(f)=\set{L\in\Nef(S)\,\vert\, f(L)<\sqrt{L^2}}
\ee
is an open and convex subcone of $\Nef(S)$.
We call $\SC(f)$ the \textit{submaximality cone} of $f$.
\end{lemma}
%\begin{proof}
%Clearly, the set $\SC(f)$ is an open cone in $\Nef(S)$.
%For the convexity consider two line bundles $L_1,L_2\in \SC(f)$ and $\lambda\in (0,1)$.
%Using the hodge index inequality we obtain
%\be 
%f\left(\lambda L_1+(1-\lambda)L_2\right)&=& \lambda f(L_1) + (1-\lambda)f(L_2)\vspace{0.2cm}\\
%&< & \lambda\sqrt{L_1^2}+(1-\lambda)\sqrt{L_2^2}\leq \sqrt{(\lambda L_1+(1-\lambda)L_2)^2}\,
%\ee
%and this shows that $\lambda L_1+(1-\lambda)L_2$ is an element of $\SC(f)$.
%\end{proof}

\begin{lemma}\label{prop:infimum-continuous}
Let $W\subset \R^g$ be a closed convex set and $M\in\R$.
Let ${\cal F}$ be a non-empty family of linear functions such that
\be 
{\cal F}\subset \set{\ell:\R^g\to \R\,\,\vert\,\, \ell \mbox{ is linear and } \ell(w)\geq M \mbox{ for all } w\in W }\,.
\ee
Then the function
\be 
f:W\to \R\,,\qquad w\mapsto \inf\set{\ell(w)\,\vert\, \ell\in\cal F}
\ee
is homogeneous, concave and continuous.
\end{lemma}
%\begin{proof}
%It follows from the linearity and the properties of the infimum that $f$ is homogen and superadditive and, therefore, $f$ is a concave function.
%Let $w\in W$ and $(w_n)_{n\in\N}\subset W$ be a sequence converging to $w$.
%Suppose that $(f(w_n))_{n\in\N}$ does not converge to $f(w)$, then there exists an $R>0$ such that infinitly many $w_n$ either satisfy
%\be 
%f(w_n)\geq f(w)+R \qquad\mbox{ or }\qquad f(w_n)\leq f(w)-R\,.
%\ee
%Without loss of generality, we assume $f(w_n)\geq f(w)+R$ for all $w_n$.
%As $f(w)$ is the pointwise infimum of the linear functions of ${\cal F}$, there exists a linear function $\ell\in{\cal F}$ with $\ell(w)<f(w)+\frac{R}{2}$.
%Since $\ell$ is continuous function, there exists an open neighborhood $U\subset M$ of $w$ such that $\ell(v)<f(w)+R$ for all $v\in U$.
%Because of the convergence of $w_n\to w$ there exists a $N\in\N$ such that $w_n\in U$ for all $n\geq N$ and, thus, we have
%\be 
%\ell(w_n)<f(w)+R\leq f(w_n)\quad \mbox{ for all } n\geq N\,
%\ee
%which is a contradiction.
%\end{proof}
Since we can consider the Seshadri function as the pointwise infimum of linear functions
\be 
\varepsilon_x:\Nef(S)\to \R_{\geq 0}\,,\qquad L\mapsto  \inf\set{\varphi_{C,x}(L)\,\,\vert\,\, C\subset S \mbox{ curve through } x}\,,
\ee
the previous Lemma provides an alternative proof for the continuity of the Seshadri function.

Szemberg showed in \cite[Prop.~1.8]{Szemberg:Hab} that an ample line bundle $L$ has at most $\rho(S)$ many submaximal curves at any given point. 
So, for a fixed line bundle $L$, we can write the Seshadri constant as a minimum over finitely many curves.
We will show that this pointwise behavior of the Seshadri function extends to a neighborhood of $L$ whenever $\eps_x(L)<\sqrt{L^2}$ and, hence, the Seshadri function is locally a piecewise linear function.
However, in the case $\eps_x(L)=\sqrt{L^2}$, the local structure can be surprisingly complex (see \cite[Thm.~3.7, Ch.~4]{Bauer-Schmidt}).
Furthermore, we will see in Prop.~\ref{prop:piecewise-linear} that on abelian surfaces with Picard number three and four the Seshadri function has never the global structure of a piecewise linear function.
\begin{proposition}\label{prop:local-structur-seshadri-function-n-curves}
Let $S$ be a smooth projective surface with Picard number $\rho(S)$, $x\in S$ and $L\in\Nef(S)$.
Suppose that $\eps_x(L)<\sqrt{L^2}$ and let $1\leq n \leq \rho(S)$ be the number of Seshadri curves of $L$ in $x$.
Then there exists an open subcone $U\subset\Nef(S)$ at $L$ such that the Seshadri function restricted to $U$ is the minimum of $n$ linear function, i.e., we have 
\be 
\eps_x|_U:U\to \R\,,\qquad L'\mapsto \min\set{\varphi_{C_i,x}(L')\,\vert\, i=1,\dots,n}\,,
\ee
where $C_1,\dots, C_n$ are the Seshadri curves of $L$ in $x$.
\end{proposition}
\begin{proof}
It follows from \cite[Prop.~1.8]{Szemberg:Hab} that $L$ has finitely many submaximal curves $C_1,\dots, C_m$.
Let $C_1,\dots, C_n$ with $n\leq m$ be the Seshadri curves of $L$ in $x$.
For any other irreducible curve $C$ which is not a Seshadri curve we obtain the following lower bound for the Seshadri quotient of $L$ in $x$ 
\be 
r:=\min\set{\frac{L\cdot C_i}{\mult_x C_i}\,\Big\vert\, i=n+1,\dots, r}\,.
\ee 
In the case $n=m$ there exist no other submaximal curve and we set $r=\sqrt{L^2}$.
We consider the family of linear function
\be
{\cal F}_L=\set{\varphi_{C,x}:\Nef(S)\to \R\,\,\Big\vert\,\, C\subset S \mbox{ is irred. curve through } x \mbox{ and } C\neq C_i,\, i=1,\dots, n}\,.
\ee
It follows from Prop.~\ref{prop:infimum-continuous} that the pointwise infimum of the family ${\cal F}_L$ yields the continuous function
$\eps_1:\Nef(S)\to \R$.
Furthermore, we consider the continuous function
\be 
\eps_2:\Nef(S)\to \R\,,\qquad L'\mapsto \min\set{\varphi_{C_i,x}(L')\,\vert\, i=1,\dots,n}\,.
\ee
Thus, the Seshadri function is given by the pointwise minimum of $\eps_1$ and $\eps_2$.

Since we have $\eps_2(L)<\eps_1(L)$, it follows from the continuity of $\eps_1$ and $\eps_2$ that there exists an open neighborhood $V\subset\Nef(S)$ at $L$ such that $\eps_2(L')<\eps_1(L')$ for all $L'\in V$.
As the Seshadri function is homogeneous, $V$ yields an open subcone $U\subset\Nef(S)$ in which the Seshadri constant is computed by $\eps_2$.
\end{proof}

For every ample line bundle $L$ on a smooth projective surface $S$, we have the general upper bound $\varepsilon_x(L)\leq \sqrt{L^2}$.
To conclude this section, we will consider the case where we have an improved upper bound for the Seshadri constant of an ample line bundle $L$, i.e., $\varepsilon_x(L)\leq R<\sqrt{L^2}$.
The following construction will yield an open subcone on which every Seshadri curve of $L$ in $x$ is submaximal.
This will play a crucial role in Sect.~\ref{sec:algo-ample} for the computation of Seshadri constants.
\begin{proposition}\label{prop:seshadri-curve-submaximal-area}
Let $S$ be a smooth projective surface, $x\in S$ and $L\in \Amp(S)$.
Furthermore, let $0<R<\sqrt{L^2}$ be an upper bound for the Seshadri constant of $L$ in $x$.
Then there exists an open convex cone $U(L,R,x)\subset \Amp(S)$ such that every Seshadri curve $C$ of $L$ in $x$ is submaximal for every line bundle $L'\in U(L,R,x)$, i.e., $U(L,R,x)\subset\SC(\varphi_{C,x})$.
\end{proposition}

\begin{proof}
We consider a compact cross-section of the nef cone with an affine hyperplane passing through $L$.
We identify the compact cross-section $K$ as a subset of $\R^{\rho(S)-1}$ and
for $v\in R^{\rho(S)-1}$ we denote the line bundle associated to $v$ by $M_v$.
Thus, we write
\be 
\varepsilon_x:K\to \R_{\geq 0}\,,\qquad v\mapsto \varepsilon_x(M_v)\,.
\ee
We denote by $M_w$ the line bundle corresponding to $L$.
We restrict the function $\eps_x$ to the line $l_i:=\set{w+te_i\,\with\, t\in\R }$.
This yields the concave function
\be 
\eps_{x,i}:[a_i,b_i]:=K\cap l_i\to \R_{\geq 0}\,,\qquad t\mapsto \varepsilon_x(M_{w+te_i})\,.
\ee
Let $C\subset S$ be a Seshadri curve of $L$ in $x$.
Then $C$ defines the affine linear function
\be 
\varphi_{C,x}:[a_i,b_i]\to \R_{\geq 0}\,,\qquad t\mapsto \frac{M_{w+te_i}\cdot C}{\mult_x C}\,.
\ee
We derive two constraints for the linear function $\varphi_{C,x}\,$:
On one hand, the linear function satisfies $\varphi_{C,x}(0)\leq R$ because it is an upper bound for the Seshadri constant of $L$. 
On the other hand, we have
\be 
\varphi_{C,x}(t)\geq 0\quad \mbox{ for all } t\in [a_i,b_i]
\ee
since the intersection number of $C$ with any nef class is non-negative.

Considering all linear functions satisfying these two conditions, we find that the affine linear functions $g_1$ and $g_2$ defined by the points $(a_i,0)$ and $(0,R)$ and, respectively, $(0,R)$ and $(b_i,0)$ provide the greatest constraint on the interval in which all Seshadri curves are submaximal.
The intersection points of $g_1$ and $g_2$ with the upper bound $t\mapsto \sqrt{M_{w+te_i}^2}$ yield an interval $(t_1,t_2)$ where every Seshadri curve is submaximal.
We illustrate the situation in Fig.~\ref{fig:submaximality-area}.
The constraints ensure that the graph of the function $\varphi_{C,x}$ is inside the gray area and, thus, every Seshadri curve is submaximal on the interval $(t_1,t_2)$.

Applying this construction to each line $l_i$ for $i=1,\dots, \rho(S)-1$ we get $2\rho(S)-2$ points in $K$.
Since $\SC(\varphi_{C_x})$ is a convex cone, it follows that every Seshadri curve is submaximal on the open convex cone $U(L,R,x)\subset\Amp(S)$ which is generated by the $2\rho(S)-2$ points.

\begin{figure}[hbt!]
\center
\begin{tikzpicture}[x=6.0cm,y=2cm]
\draw (-0.75,0.) -- (0.6,0.);
\draw[shift={(-0.75,0)},color=black] (0pt,2pt) -- (0pt,-2pt) node[below] {\footnotesize $a_j$};
\draw[shift={(0.6,0)},color=black] (0pt,2pt) -- (0pt,-2pt) node[below] {\footnotesize $b_j$};

\draw[->,color=black] (-0.15,0) -- (-0.15,1.2) node[above] {$\eps(t)$};
\draw[shift={(0.225,0.95)},color=black] (0pt,3pt) -- (0pt,3pt) node[above] {\footnotesize $t\mapsto \sqrt{M_{w+te_i}^2}$};
\draw[dashed, domain=-0.75:0.6, samples=300] plot (\x, {sqrt(1-1*\x*\x)});

\draw[shift={(-0.15,0.6)},color=black] (2pt,0pt) -- (-2pt,0pt) node[left] {\footnotesize $R$};

\draw (0.45,.3) node {$g_2$};
\draw[color=black](0.6,0)--(-0.4898,0.871837);

\draw (-0.58,.3) node {$g_1$};
\draw[color=black](-0.75,0.)--(0.22448,0.97448);

\draw[dotted](-0.4898,0)--(-0.4898,0.871837);
\draw[shift={(-0.4898,0)},color=black] (0pt,2pt) -- (0pt,-2pt) node[below] {\footnotesize $t_1$};

\draw[dotted](0.22448,0)--(0.22448,0.97448);
\draw[shift={(0.22448,0)},color=black] (0pt,2pt) -- (0pt,-2pt) node[below] {\footnotesize $t_2$};

\draw[shift={(-0.15,0)},color=black] (0pt,2pt) -- (0pt,-2pt) node[below] {\footnotesize $0$};

%\fill[gray!40,nearly transparent] (-0.4898,0) -- (0.22448,0) -- (0.22448,0.97448) -- (-0.15,0.6) -- (-0.4898,0.871837) -- cycle;

\fill[gray!40,nearly transparent] (-0.75,0) -- (0.6,0) -- (0.6,1.35) -- (-0.15,0.6) -- (-0.75,1.08) -- cycle;

\end{tikzpicture}
\caption{\label{fig:submaximality-area} Minimal interval $(t_1,t_2)$ on which every Seshadri curve is submaximal.}
\end{figure}
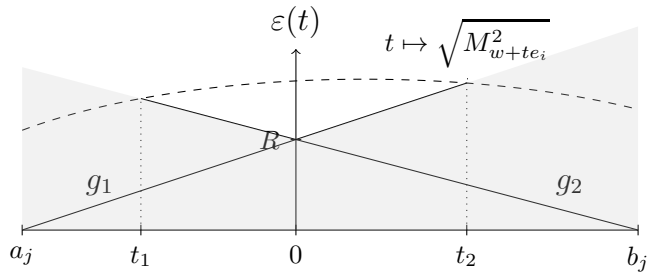

\end{proof}

%*****************************************************************************

\section{ Submaximal curves and Pell bounds on abelian surfaces}
\label{sec:submaximal-curves-Pell-bounds}
On abelian surfaces, every non-elliptic curve is ample. 
We can therefore write $\varepsilon(L)$ as
\be
\varepsilon(L)=\min\Big\{\varepsilon_{\rm ell}(L),\varepsilon_{\rm amp}(L),\sqrt{L^2}\Big\}\,,
\ee
where $\varepsilon_{\rm ell}$ denotes the infimum of all \textit{submaximal elliptic} curves passing through $0$ and, respectively, $\varepsilon_{\rm amp}$ denotes the infimum of all \textit{submaximal ample} curves.

\subsection{ Elliptic curves}
In the following, we will clarify the role of elliptic curves in the computation of Seshadri constants.
In \cite{Kani:elliptic} Kani has shown that we can detect elliptic curves solely by their numerical class:
\begin{proposition}[{\cite[Prop.~2.3]{Kani:elliptic}}]\label{prop:crit-elli-2}
Let $A$ be an abelian surface and $D\in\NS(A)$.
Then $D\equiv mE$ for an elliptic curve $E$ and $m\in \Z$ if and only if $D^2=0$.
Moreover, we have $m>0$ if and only if the intersection $H\cdot D$ with an ample line bundle $H$ is positive.
\end{proposition}
This means that the primitive elements of the Néron-Severi group on the boundary of the nef cone are classes of elliptic curves.
Therefore, we can describe elliptic curves and $\varepsilon_{\rm ell}$ only by the numerical data of the Néron-Severi group.

Elliptic curves play a crucial role in the computation of Seshadri constants.
As it turns out, every elliptic curve passing through $0$ computes the Seshadri constants in an open subcone and, therefore, contributes to the Seshadri function.
\begin{proposition}\label{prop:elliptic-seshadri-curve-on-cone}
Let $A$ be an abelian surface and $E$ an elliptic curve passing through $0$.
Then we have:
\begin{itemize}\compact
\item[(i)] $E$ computes the Seshadri constant of $\O_A(E)$ and $E$ is the only weakly-submaximal curve of $\O_A(E)$.
\item[(ii)] There exists an open subcone in $U\subset \Amp(A)$ such that $E$ computes the Seshadri constant for every $L\in U$.
\end{itemize}
\end{proposition}
\begin{proof}
The first statement follows immediately from the fact that $E^2=0$ and that for every other curve $C$ through $0$ we have $E\cdot C>0$.

For the second, more subtle statement, we claim that there exists an ample line bundle $L$ such that $E$ is the only Seshadri curve of $L$.
Then, by applying Prop.~\ref{prop:local-structur-seshadri-function-n-curves} we obtain an open subcone $U$ at $L$ where $E$ computes the Seshadri constants.

To find such a line bundle $L$, we start with an ample line bundle $H$ which satisfies $(H\cdot E)^2 < H^2$.
(Note that for any given ample line bundle $\tilde{H}$ we can chose the line bundle $H:=\tilde{H}+(\tilde{H}\cdot E)\O_A(E)$.)
We consider the Seshadri function restricted to the ray $\set{H_t:=E+tH\,\with\, t\geq 0}\subset\Nef(A)$ and write
\be 
\eps:[0,\infty)\to \R\,,\qquad t\mapsto \eps(H_t)\,,
\ee
and, respectively, 
\be
\varphi_C:[0,\infty)\to \R,\qquad t\mapsto \frac{H_t\cdot C}{\mult_0 C}
\ee
for the linear function of a curve $C$ restricted to the ray.

It follows from the choice of $H$ that $E$ is submaximal for $H_1$.
Moreover, since the upper bound $t\mapsto \sqrt{H_t^2}$ is concave, it follows that the $E$ is submaximal for all $H_t$ with $t\in (0,1]$.

We claim that there are finitely many Seshadri curves computing the Seshadri constants for all $H_s$ with $s\in (0,1)$.
For this, we show that every Seshadri curve $C$ of $H_s$ with $s\in (0,1)$ is also submaximal for $H_1$.
But by Szemberg \cite[Prop.~1.8]{Szemberg:Hab} the line bundle $H_1$ has at most $\rho(A)$ submaximal curves and, thus, there can only be finitely many Seshadri curves as claimed.
To show that a Seshadri curve $C$ of $H_s$ with $s\in (0,1)$ is also submaximal for $H_1$, we will compare the slopes of the linear functions $\varphi_E$ and $\varphi_C$.
We know that $E$ computes the Seshadri constant of $H_0$ and that $C$ computes the Seshadri constant of $H_s$.
Therefore, we have
\be 
\varphi_E(0)\leq \varphi_C(0) \qquad \mbox{and} \qquad \varphi_E(s) \geq \varphi_C(s)\,.
\ee
This shows that the slope of $\varphi_C$ is less than or equal to the slope of $\varphi_E$ and, hence, we have $\varphi_E(1)\geq  \varphi_C(1)$.
Since $E$ is submaximal for $H_1$, it follows that $C$ is also submaximal for $H_1$ as claimed.
We denote by $C_1,\dots, C_n$ the Seshadri curves which compute the Seshadri constants for all $H_s$ with $s\in (0,1)$.

Lastly, we claim that we have $E=C_i$ for an $i\in\set{1,\dots,n}$ and there even exists a line bundle $H_s$ such that $E$ is the only Seshadri curve.
We assume to the contrary that $C_i\neq E$ for all $i\in\set{1,\dots,n}$.
It follows from (i) that we have
\be 
\varepsilon(H_0)=0=\sqrt{H_0^2}<\varphi_{C_i}(0)=\frac{\O_A(E)\cdot C_i}{\mult_0 C}\qquad \mbox{ for all } i\in\set{1,\dots,n}\,.
\ee
Since the functions $t\mapsto \sqrt{H_t^2}$ and $t\mapsto \min\set{\varphi_{C_i}(t)\,\with\, i=1,\dots,n}$ are continuous, there exists a number $t_0\in(0,1)$ such that none of the curves $C_1,\dots, C_n$ are submaximal for any line bundle $H_s$ with $s\in [0,t_0]$.
This, however, is a contradiction because $C_1,\dots, C_n$ compute the Seshadri constants for all $H_s$ with $s\in (0,1)$.
Thus, the assumption was incorrect and we have $E=C_j$ for a $j\in\set{1,\dots,n}$.
Now, by applying the exact same argument as before to the remaining curves $C_i$ with $i\neq j$, we again find a number $t_0\in (0,1)$ such that none of the curves $C_i$ with $i\neq j$ are submaximal for line bundles $H_s$ with $s\in [0,t_0]$.
Therefore, $E$ is the only Seshadri curve for $H_s$ with $s\in [0,t_0]$.

Thus, we have found ample line bundles $H_s$ such that $E$ is the only Seshadri curve.
It follows from Prop.~\ref{prop:local-structur-seshadri-function-n-curves} that there exists an open subcone at $H_s$ in which $E$ computes the Seshadri constants.
\end{proof}

\subsection{Submaximal ample curves and Pell bounds}
In this section, we show that $\varepsilon_{\rm amp}$ can be computed from Pell bounds, which are constructed purely numerical via Pell equations.
This implies that Seshadri constants on abelian surfaces are a purely numerical problem since $\varepsilon_{\rm ell}$ and $\varepsilon_{\rm amp}$ only depend on the numerical data given by the intersection matrix. 
We recall the definition of Pell bounds.
\begin{definition}
Let $A$ be an abelian surface and $L$ an ample primitive symmetric line bundle such that $\sqrt{L^2}\notin\Z$.
Consider the Pell equation
\be
x^2 - L^2\cdot y^2 = 1
\ee
and denote by $(\ell,k)$ its primitive solution.
Then there exists a divisor $D\in|2kL|^+$ with $\mult_0 D\ge 2\ell$ (see \cite[Thm.~A.1]{Bauer-Szemberg:periods-appendix}).
Any such divisor is called a \emph{Pell divisor} for $L$ and the existence of Pell divisors of $L$ yields the linear function 
\be 
\pi_L:\NS_\R(A)\to \R,\qquad M\mapsto \frac{M\cdot(kL)}{\ell}
\ee
which is locally an improvement of the general upper bound $\sqrt{L^2}$ for the Seshadri constant.
We call $\pi_L$ the \emph{Pell bound} of $L$.

A Pell divisor or Pell bound of $qL$ with $q\in\Q^+$ is by definition Pell bound or Pell divisor or Pell bound of $L$.
\end{definition}

One of the main ingredients for computing Seshadri constants will be the result of \cite[Prop.~1.3]{Bauer-Schmidt} which essentially provides us with the numerical data of submaximal ample curves.
We summarize the result and connect it to Pell bounds.
\begin{proposition}[{\cite[Prop.~1.3]{Bauer-Schmidt}}]\label{prop:old-result}
Let $A$ be an abelian surface and $C\subset A$ a submaximal ample curve.
Then $C$ or $2C$ appears as a unique Pell divisor of $\O_A(C)$ and the numerical data of $C$ is given by the primitive solution of the Pell equation, i.e., the linear function $\varphi_C$ coincides with the Pell bound $\pi_{\O_A(C)}$.
\end{proposition}

Therefore, we can estimate $\varepsilon_{\rm amp}$ in terms of Pell bounds:
\begin{corollary}\label{cor:epsamp-and-pellbounds}
Let $A$ be an abelian surface and $L$ an ample line bundle. Then we have
\be 
\varepsilon(L)\leq \inf\set{\pi_P(L)\,\with\, \pi_P \mbox{ is Pell bound of }P\in\NS(A)}\leq \varepsilon_{\rm amp}(L)
\ee
where equality on both sides holds if and only if $L$ has a submaximal ample Seshadri curve.
\end{corollary}
The first inequality follows immediately from the existence of Pell divisors and the second inequality follows from Prop.~\ref{prop:old-result}.

Since Pell bounds and elliptic curves only depend on the numerical data given by the Néron-Severi group, it follows from the representation
\be
\varepsilon(L)=\min\Big\{\varepsilon_{\rm ell}(L),\varepsilon_{\rm amp}(L),\sqrt{L^2}\Big\}
\ee
that Seshadri constants on abelian surfaces are a purely numerical problem only dependent on the numerical data provided by the Néron-Severi group.
Thus, we have proven Cor.~\ref{introcor:seshadri-numerical} in the introduction.

\subsection{Finding submaximal ample curves via Pell bounds}
In this section, we provide a method to determine whether or not a given Pell bound coincides with a linear function of a submaximal ample curve.
This generalizes the result in \cite[Prop.~3.13]{Bauer-Schmidt} for the special case of principally polarized abelian surfaces with real multiplication to all abelian surfaces, and is used in Sect.~\ref{sec:algo-ample} to compute ample Seshadri curves.
The main problem in extending the statement to all abelian surfaces is the consideration of elliptic curves and how they affect the method.

First, we claim that Pell bounds of different primitive ample classes always define different linear functions.
\begin{lemma}\label{lem:different-pell-bound}
Let $A$ be an abelian surface with Picard number $\rho$, and let $L$ and $M$ be ample primitive symmetric line bundles with $\sqrt{L^2}\,,\sqrt{M^2}\notin\Z$.
Then the linear functions $\pi_L$ and $\pi_M$ coincides if and only if $L\equiv M$.
\end{lemma}
\begin{proof}
If $L\equiv M$ holds, then the Pell bounds clearly coincide.

For the converse, let $D_1,\dots, D_{\rho}$ be a basis of the Néron-Severi group with intersection matrix $S$.
We write $L\equiv\sum_{i=1}^{\rho}a_i D_i$ and $M\equiv \sum_{i=1}^{\rho}b_i D_i$.
Then, the linear functions can be written as
\be 
\pi_L\left(\sum_{i=1}^{\rho}r_iD_i\right)= 
\frac{\left( \sum_{i=1}^{\rho}r_iD_i \right)\cdot (k_LL)}{\ell_L}=
\frac{k_L}{\ell_L}\left(a_1, \dots, a_{\rho}\right)
S 
\left(r_1,\dots,r_{\rho}\right)^t
\ee
and 
\be 
\pi_M\left(\sum_{i=1}^{\rho}r_iD_i\right)= 
\frac{\left( \sum_{i=1}^{\rho}r_iD_i \right)\cdot \left(k_MM\right)}{\ell_M}=
\frac{k_M}{\ell_M}\left(b_1, \dots, b_{\rho}\right)
S 
\left(r_1,\dots,r_{\rho}\right)^t
\ee
where $(\ell_L, k_L)$ and $(\ell_M,k_M)$ denote the primitive solutions of the Pell equations given by $x^2-L^2y^2=1$ and $x^2-M^2y^2=1$.
Since $S$ is invertible, there are vectors $v_i\in\R^\rho$ such that $Sv_i=e_i$ for $i=1,\dots, \rho$.
By assumption $\pi_L$ and $\pi_M$ coincide and, thus, we obtain the following identity for the coefficients
\be 
b_i = \frac{\ell_M \, k_L}{k_M \, \ell_L}\cdot a_i \qquad \mbox{ for all } i=1,\dots, \rho\,.
\ee
This shows that the numerical class of $M$ is a rational multiple of the class of $L$.
Since $L$ and $M$ are primitive we have $L\equiv M$. 
\end{proof}
Since the linear function $\varphi_C$ of a submaximal ample curve coincides with the Pell bound $\pi_{\O_A(C)}$, the previous result extends to submaximal curves.
\begin{corollary}\label{cor:submax-and-pell}
Let $A$ be an abelian surface with Picard number $\rho$, $L$ an ample primitive symmetric line bundle with $\sqrt{L^2}\notin\Z$ and $C$ a submaximal ample curve.
Then the linear function $\varphi_C$ and $\pi_L$ coincides if and only if $L\equiv r\O_A(C)$ with $r\in\Q^+$.
\end{corollary}

The following result provides a method to find ample Seshadri curves through Pell bounds.
This will be crucial to not only compute single Seshadri constants but to also compute Seshadri curves.
By Prop.~\ref{prop:local-structur-seshadri-function-n-curves} this will give us access to the local structure of Seshadri functions and even plot them.
\begin{theorem}\label{thm:pell-bound-submaximal-curve-1-to-1}
Let $A$ be an abelian surface and let $L$ be an ample primitive symmetric line bundle with $\sqrt{L^2}\notin\Z$.
Then we have the following equivalence:
\begin{enumerate}\compact
\item[\rm(i)] There exists an irreducible curve $C$ such that $C$ or $2C$ appears as a unique Pell divisor of $L$ and the numerical data of $C$ is given by the primitive solution of the Pell equation.
\item[\rm(ii)] There exists an irreducible curve $C\in |nL|$ for an $n\in\N$ which is submaximal for $L$.
\item[\rm(iii)] The Pell bound $\pi_L$ coincides with the Seshadri function in an open subcone $U\subset\Nef(A)$ containing $L$.
\item[\rm(iv)] For any other ample primitive symmetric line bundle $M$ with $\sqrt{M^2}\notin\Z$ and every elliptic curve $E$ we have $\pi_L(L)<\pi_M(L)$ and $\pi_L(L)<\varphi_E(L)$.
\item[\rm(v)] The Pell bound $\pi_L$ coincides with a linear function $\varphi_C$ where $C$ is a submaximal ample curve.
\end{enumerate}
\end{theorem}
In other words, (vi) shows that a Pell bound $\pi_L$ yields a submaximal ample curve if and only if $\pi_L$ takes an absolute minimum in the set of Pell bounds and linear functions given by elliptic curves.
This criterion will allow us to decide numerically whether or not a Pell bound corresponds to a submaximal curve and provides a method to compute the numerical data of the Seshadri curve.

\begin{proof}
The equivalence of (i) and (ii) is a consequence of Prop.~\ref{prop:old-result}.

Assume that (ii) holds.
The result of Cor.~\ref{cor:submax-and-pell} shows that $\varphi_C$ and $\pi_L$ coincides.
Furthermore, it follows from \cite[Lemma~5.2]{Bauer:sesh-alg-sf} that $C$ is the only submaximal curve of $L$ and, therefore, the Seshadri function is computed by $\varphi_C$ in an open subcone by Prop.~\ref{prop:local-structur-seshadri-function-n-curves}.

Next, assume that (iii) holds. We have to show that there does not exist another Pell bound $\pi_{M}$ with $\pi_M(L)=\pi_L(L)$ or an elliptic curve with $\varphi_E(L)=\pi_L(L)$.
Suppose to the contrary that equality holds in one of the cases.
It follows from Lemma~\ref{lem:different-pell-bound} that the Pell bounds $\pi_L$ and $\pi_M$ do not coincide globally.
Furthermore, a Pell bound can never coincide with the the function $\varphi_E$ of an elliptic curve $E$, because we have $\varphi_E(\O_A(E))=E^2=0$ and $\pi_L(\O_A(E))=(E\cdot kL)/\ell>0$ since $kL$ is ample.
Thus, $\pi_L$ and $\varphi_E$ do not coincide globally.
Since the functions are linear, they can only differ, if there exists a direction such that the directional derivative differs.
This, however, yields a line bundle $L'\in U$ with $\pi_{M}(L')<\pi_L(L')$ or $\varphi_E(L')<\pi_L(L')$, which is a contradiction to (iii).

Now, assume that (iv) holds and let $C$ be a Seshadri curve of $L$.
By assumption every elliptic curve satisfies $\varepsilon(L)\leq \pi_L(L)<\varphi_E(L)$ and, thus, $C$ is a submaximal ample curve.
We denote by $\pi_{L'}$ the unique Pell bound with $\pi_{L'}=\varphi_C$.
Since $C$ and, therefore, $\pi_{L'}$ computes the Seshadri constant of $L$, we have for every other Pell bound $\pi_M$
\be 
\varepsilon(L)=\varphi_C(L)=\pi_{L'}(L)\leq \pi_M(L)\,.
\ee
From our assumption, it follows that $\pi_{L'}=\pi_L$ and, thus, we have shown that $\pi_L$ coincides with the linear function $\varphi_C$.

Lastly, we show the implication (v) $\Rightarrow$ (ii).
We assume that $\varphi_C=\pi_L$ for a submaximal ample curve.
Since we have $\pi_L(L)<\sqrt{L^2}$ it follows that $C$ is submaximal for $L$.
Let $P\in |2kL|$ be a Pell divisor of $L$ with $\mult_0 P\geq 2\ell$.
By definition of the Pell bounds we have 
\be 
\varphi_P(M)=\frac{M\cdot P}{\mult_0 P}\leq \frac{M\cdot (kL)}{\ell}= \pi_L(M)=\varphi_C(M)\qquad \mbox{ for all } M\in\Nef(A)\,.
\ee
This shows that $P$ computes the Seshadri constant of $\O_A(C)$.
Since $C$ is the only submaximal curve of $\O_A(C)$ it follows from the subsequent Lemma that $P=mC$ for some $m\in\N$.

Thus, we have shown that every Pell divisor of $L$ is a multiple of $C$ and, therefore, there is an $n\in\N$ such that $C\in |nL|$.
\end{proof}

\begin{lemma}
Let $A$ be an abelian surface and $L$ an ample line bundle.
Suppose that $D$ is an effective divisor which computes the Seshadri constant of $L$.
Then every component of $D$ passes through $0$ and computes the Seshadri constant of $L$.
\end{lemma}
\begin{proof}
Let $C$ be an irreducible component of $D$.
Assume that $\mult_0 C=0$, then we write $D=C+R$ for an effective divisor $R$ with $\mult_0 R=\mult_0 D$.
Since the intersection number of $C$ with $L$ is strictly positive, we get
\be
\varepsilon(L)=\frac{L\cdot D}{\mult_0 D}=\frac{L\cdot (C+R)}{\mult_0 R}> \frac{L\cdot R}{\mult_0 R}\,.
\ee
which is a contradiction, since $D$ computes the Seshadri constant of $L$.

Now, suppose to the contrary that there exists an irreducible component $C$ of $D$ which does not compute the Seshadri constant and, thus, we have $\varphi_D(L)<\varphi_{C}(L)$.
Again, we write $D=C+R$ for an effective divisor $R$ and get
\be
\varepsilon(L)=\frac{L\cdot (C+R)}{\mult_0 C + \mult_0 R}=\varphi_D(L)<\varphi_C(L)=\frac{L\cdot C}{\mult_0 C}\,.
\ee
But this implies that
\be
\varepsilon(L) = \frac{L \cdot D}{\mult_0 D} > \frac{L\cdot R}{\mult_0 R}\,,
\ee
which gives the same contradiction as before.
\end{proof}

With Thm.~\ref{thm:pell-bound-submaximal-curve-1-to-1}~(iv) we can find submaximal curves by using isometries.
\begin{corollary}\label{cor:seshadricurves-isometry}
Let $A$ be an abelian surface and let $\psi:\NS(A)\to \NS(A)$ be an isometry with respect to the intersection product that leaves the nef cone invariant.
Then, a curve $C$ is an irreducible submaximal curve if and only if there exists an irreducible submaximal curve $C'$ such that $C'\equiv \psi(\O_A(C))$ and $\mult_0 C'=\mult_0 C$.
\end{corollary}
\begin{proof}
If $C$ is an elliptic curve then $\psi(\O_A(C))$ is a primitive nef line bundle satisfying Prop.~\ref{prop:crit-elli-2} and, thus, yields an elliptic curve.
If $C$ is a submaximal ample curve, then the Pell bound of $\psi(\O_A(C))$ has the same numerical properties as the Pell bound of $\O_A(C)$ and, therefore, we can apply Thm.~\ref{thm:pell-bound-submaximal-curve-1-to-1}~(iv).
\end{proof}

\section{Computation of Seshadri constants on abelian surfaces}
\label{sec:computing-seshc}
In this section, we will distinguish between two cases, namely whether the Seshadri constant of an ample line bundle is computed by an elliptic or an ample curve.
For both cases, we will develop a method to effectively compute the Seshadri constant on any ample line bundle on abelian surfaces if the numerical data (intersection matrix) of the Néron-Severi group is given.

The method for computing Seshadri constants on principally polarized abelian surfaces with real multiplication in \cite{Bauer-Schmidt} depends heavily on the simple and known structure of the intersection matrices.
We, therefore, begin by establishing a standardized basis of the Néron-Severi vector space together with a cross-section of the nef cone in which we can perform the calculations without any reference to specific intersection matrices.

Assume that $D_0,\dots, D_{\rho-1}$ is a basis of $\NS(A)$ and that the orientation of the nef cone is known.
The Hodge index theorem yields a base change of the Néron-Severi vector space such that $B_0,\dots, B_{\rho-1}$ is a $\R$-basis of $\NS_\R(A)$ where the intersection matrix is given by
\be
\left(
\begin{array}{cccc}
	1 & 0 & \cdots & 0  \\
	0 & -1 & \cdots  & 0  \\
	\vdots & \vdots & \ddots & \vdots\\
	0 & 0 & \dots & -1
	\end{array}
\right)\in\Z^{\rho\times\rho}\,.
\ee
By replacing $B_0$ with $-B_0$, we can further assume that $B_0$ is ample, which fixes the orientation of the nef cone.
This yields a standardized compact cross-section of the nef cone given by
\be 
\mathcal N(A):=\set{1\cdot B_0 + \sum_{i=1}^{\rho-1} r_iB_i \in\NS_\R(A)\with 1-\sum_{i=1}^{\rho-1} r_i^2\geq 0}\,,
\ee
which corresponds to the unit ball in $\R^{\rho-1}$.

By using a $\R$-valued base change of the Néron-Severi vector space we lose information about the Néron-Severi group.
We will show in Remark~\ref{rem:eps-elliptic} how this information can be recovered.

\subsection{Seshadri constant computed by an elliptic Seshadri curve}
\label{sec:sesh-elliptic-case}
We will first consider the case of an ample line bundle $L\in\NS_\R(A)$ whose Seshadri constant is computed by an elliptic curve.
Finding the Seshadri constant in this case amounts to computing the minimal degree of elliptic curves with respect to $L$.
In Picard number two, this is an easy task, since there are at most two elliptic curves.
However, on non-simple abelian surfaces with Picard number three and four, there are infinitely many elliptic curves and until now it was unclear how to find the minimal degree and the corresponding elliptic curves.
We show how to reduce the number of elliptic curves we have to consider to a finite number, and by taking the minimum over the finite set of elliptic curves, we can compute the Seshadri constant of $L$ and find all elliptic Seshadri curves.

We consider the cross-section $\mathcal N(A)$ of the previous section.
Let $L=\sum_{i=0}^{\rho-1} a_i B_i$ be the representation of $L$ in the basis $B_0,\dots,B_{\rho-1}$.
Then the unique representative of $L$ in $\mathcal N(A)$ is given by $L'=1\cdot B_0 + \sum_{i=1}^{\rho-1} \frac{a_i}{a_0} B_i$.
Since elliptic curves are on the boundary of the nef cone, every elliptic curve $E\subset A$ has a unique representative in $\partial \mathcal N(A)$.
In order to find the smallest intersection number of $L$ with an elliptic curve, we first compute the smallest intersection of $L'$ with \textit{any} line bundle on $\partial \mathcal N(A)$, i.e.,
\be 
r:=\min\set{L'\cdot M\,\with\, M\in\partial\mathcal N(A)}=\min\set{
1-v^T
\left(
\begin{array}{c}
	\frac{a_1}{a_0} \\
	\vdots \\
	\frac{a_{\rho-1}}{a_0}
\end{array}
\right)
\,\with\, v\in S^{\rho-1}}\,.
\ee
This minimum can be computed explicitly (for example with the method of Lagrange multipliers) and is given by
\be 
r=1-\frac{\sqrt{\sum_{i=1}^{\rho-1}a_i^2}}{a_0}\,. 
\ee
Let $E'\in\mathcal N(A)$ be a representative of an elliptic curve $E$ and let $\lambda\in\R_{>0}$ be the number with $E=\lambda E'$.
Then we have
\be 
L\cdot E = (a_0L')\cdot (\lambda E')\geq a_0\lambda r\,.
\ee
Since we focus on elliptic curves which are submaximal for $L$, it follows that $\lambda$ is bounded as follows
\be 
\lambda\leq \frac{\sqrt{L^2}}{a_0r}=\frac{\sqrt{a_0^2-\sum_{i=1}^{\rho-1}a_i^2}}{a_0-\sqrt{\sum_{i=1}^{\rho-1}a_i^2}}\,.
\ee
Therefore, every elliptic curve which is submaximal for $L$ is contained in the compact set
\be 
W_{\rm ell}:=\left[0,\frac{\sqrt{L^2}}{a_0r}\right]
\times 
\left[
-\frac{\sqrt{L^2}}{a_0r},
\frac{\sqrt{L^2}}{a_0r}
\right]^{\rho-1}\,.
\ee
Since the Néron-Severi group is a lattice in $\NS_\R(A)$, there only exist finitely many classes in $W_{\rm ell}\cap \NS(A)$ and the Seshadri constant is given by 
\be 
\varepsilon(L)=\min\set{L\cdot E \,\with\, E\in W_{\rm ell}\cap \NS(A)\setminus\set{0} \mbox{ and } E^2=0}\,.
\ee
Furthermore, every class $E\in W_{\rm ell}\cap \NS(A)$ with $\varepsilon(L)=L\cdot E$ yields a unique elliptic curve passing through $0$ computing the Seshadri constant.

\begin{remark}\label{rem:proper-lb-basechange}
In the previous discussion, it is a priori unclear which elements of $W_{\rm ell}$ belong to the Néron-Severi group because that information got lost by using an $\R$-valued change of bases.
We can resolve this issue by reverting the base change we made at the beginning of the section.
Let $T=(t_{i,j})$ be the base change from $B_0,\dots, B_{\rho-1}\in\NS_\R(A)$ to the base $D_0,\dots, D_{\rho-1}\in\NS(A)$, then the image $T(E)$ is contained in the cube $[-m_{T},m_{T}]^{\rho}$ with
\be 
m_{T}:=\max\set{\frac{\sqrt{L^2}}{a_0r}\cdot \sum_{j=1}^{\rho-1}|t_{i,j}|
\with j=1,\dots,\rho-1}\,.
\ee
Again, there are only finitely many elements to consider, but in this case, we exactly know which elements are in the Néron-Severi group.
\end{remark}

The previous considerations yield the first part of Thm.~\ref{introthm:seshadri-algorithm}:
\begin{theorem}\label{thm:elliptic-sesh}
Let $L$ be an ample line bundle and suppose that $L$ has an elliptic Seshadri curve.
Then there exists an effective algorithm to compute $\varepsilon(L)$ and the numerical data of all elliptic Seshadri curves of $L$.
\end{theorem}

\begin{remark}\label{rem:eps-elliptic}\rm
Using the previous method, we can also explicitly compute the minimal degree of elliptic curves with respect to $L$, even when the Seshadri constant of $L$ is not computed by an elliptic curve.
This means, we can compute $\varepsilon_{\rm ell}(L)$ for any ample line bundle by considering the compact set given by
\be 
W_{\rm ell}':=\left[0,\frac{M}{a_0r}\right]
\times 
\left[
-\frac{M}{a_0r},
\frac{M}{a_0r}
\right]^{\rho-1}\,.
\ee
where $M>0$ has to be sufficient large so that at least one elliptic curve is in $W_{\rm ell}'$, e.g. by putting $M=L\cdot E$ for an elliptic curve $E\subset A$.
This computes the minimal degree of elliptic curves with respect to $L$.
\end{remark}

\subsection{Seshadri constant computed by an ample Seshadri curve}\label{sec:algo-ample}
In this section, we will consider the case of an ample line bundle $L\in\NS_\Q(A)$ whose Seshadri constant is computed by an ample curve.
(Note that for the method presented, it is necessary that $L$ is a rational class.)
We will further assume that $\varepsilon(L)<\sqrt{L^2}$ and, thus, the ample Seshadri curves are submaximal.
In Remark~\ref{rem:seshadri-takes-upper-bound} we show how to determine whether $\varepsilon(L)=\sqrt{L^2}$ holds or not.

In \cite{Bauer-Schmidt} Seshadri constants were computed in the special case of principally polarized abelian surfaces by showing that, on one hand, we found an interval $I\subset\R$ on which the Seshadri curve is submaximal, and on the other hand, we showed that there are only finitely many Pell bounds whose submaximality domain contain the interval $I$.
On arbitrary abelian surfaces, the standardized cross-section of the nef cone and the preparations made in Prop.~\ref{prop:seshadri-curve-submaximal-area} and Thm.~\ref{thm:pell-bound-submaximal-curve-1-to-1} will enable us to generalize this approach.

We start by explicitly computing the submaximality domain of Pell bounds in the standardized cross-section of the nef cone $\mathcal N(A)$.
We assume that $P=\sum_{i=0}^{\rho-1} p_i B_i\in\NS(A)$ is an ample primitive symmetric line bundle such that the self-intersection $P^2$ is not a square number, and denote by $(\ell,k)$ the primitive solution of the Pell equation $x^2-P^2y^2=1$.
The Pell bound of $P$ restricted to the cross-section $\mathcal N(A)$ is given by
\be 
\pi_P:\mathcal N(A) \to \R_{\geq 0},\qquad M=1\cdot B_0 + \sum_{i=1}^{\rho-1} r_iB_i\,\mapsto\, \frac{M\cdot (kP)}{\ell}\,=\,\frac{k}{\ell}\left(p_0+\sum_{i=1}^{\rho-1}p_ir_i\right)\,.
\ee
We will estimate the volume of the \textit{submaximality domain} in $\mathcal N(A)$, i.e,
\be 
\SD(\pi_P):=\set{M\in \mathcal N(A)\with \pi_P(M)<\sqrt{M^2}}=\SC(\pi_P)\cap \mathcal N(A)\,.
\ee

\begin{proposition}\label{prop:volume}
Let $P=\sum_{i=0}^{\rho-1} p_i B_i\in\NS(A)$ be an ample primitive symmetric line bundle such that $\sqrt{P^2}\notin\Z$.
Then the volume of ${\SD}(\pi_P)$ is bounded as follows
\be
\Vol({\SD}(\pi_P))<\frac{\Vol(S^{\rho-2})}{p_0^{\rho-1}}\,.
\ee
\end{proposition}
\begin{proof}
A class $M=1\cdot B_0 + \sum_{i=1}^{\rho-1} r_iB_i\in\mathcal N(A)$ is a boundary point of the domain ${\SD}(\pi_P)$ if and only if it satisfies the equation
\be 
f(r_1,\dots,r_{\rho-1}):=\pi_P(M)^2-(\sqrt{M^2})^2=\frac{k^2}{\ell^2}\left(p_0-\sum_{i=1}^{\rho-1} p_ir_i \right)^2 - \left(1-\sum_{i=1}^{\rho-1} r_i^2\right)=0\,,
\ee
which describes an ellipsoid in $\R^{\rho-1}$.

The next step is to find an upper bound for the volume of the ellipsoid only dependent on the numerical data given by the line bundle $P$.
For this, we are completing the square of $f$ by translating the ellipsoid by
\be 
d_i:=\frac{p_0p_i k^2}{p_0^2k^2+1}\qquad \mbox{ for } i=1,\dots, \rho-1\,.
\ee 
After some simplification using the Pell equation $\ell^2-P^2k^2=1$, this yields the ellipsoid
\begin{equation}\label{eqn:ellipsoid}
f(r_1+d_1,\dots, r_{\rho-1}+d_{\rho-1})=
\left(
\begin{array}{c}
	r_1 \\
	\vdots \\
	r_{\rho-1}
\end{array}
\right)^T
Q
\left(
\begin{array}{c}
	r_1 \\
	\vdots \\
	r_{\rho-1}
\end{array}
\right)
-\frac{1}{p_0^2k^2+1}=0\,,
\end{equation}
where the matrix $Q$ is a positive definite matrix given by
\be
Q:=(q_{i,j})=
\begin{bycases}
p_i^2\frac{k^2}{\ell^2}+1		& \mbox{ if } i=j \\
p_ip_j\frac{k^2}{\ell^2}	& \mbox{ if } i\neq j\,.
\end{bycases}
\ee
and satisfies
\be 
\det(Q)=\frac{1}{\ell^2} \left(\ell^2+\sum_{i=1}^{\rho-1}p_i^2k^2\right)\,\stackrel{\rm Pell}{=}\,\frac{p_0^2k^2+1}{\ell^2}\,.
\ee
As $Q$ is a symmetric positive definite matrix, there exists a unique positive definite matrix $\sqrt{Q}$ such that $\sqrt{Q}^2=Q$.
Thus, equation \eqnref{eqn:ellipsoid} is equivalent to
\be 
1=\left(
\begin{array}{c}
	r_1 \\
	\vdots \\
	r_{\rho-1} 
\end{array}
\right)^T
\left(\sqrt{p_0^2k^2+1}\cdot \sqrt{Q}\right)^2
\left(
\begin{array}{c}
	r_1 \\
	\vdots \\
	r_{\rho-1} 
\end{array}
\right)=
\left\Vert
\left(\sqrt{p_0^2k^2+1}\cdot \sqrt{Q}\right)
\left(
\begin{array}{c}
	r_1 \\
	\vdots \\
	r_{\rho-1} 
\end{array}
\right)
\right\Vert^2\,.
\ee
This shows that the linear function $\sqrt{p_0^2k^2+1}\cdot \sqrt{Q}$ maps the translated ellipsoid onto the unit sphere in $\R^{\rho-1}$.
The volume formula for linear maps yields
\be 
\Vol(S^{\rho-2})=\det\left(\sqrt{p_0^2k^2+1}\cdot \sqrt{Q}\right)\cdot \Vol({\SD}(\pi_P))\,,
\ee
and, therefore, we obtain
\be 
\Vol({\SD}(\pi_P))=\frac{\Vol(S^{\rho-2})\cdot \ell}{\left(\sqrt{p_0^2k^2+1}\right)^\rho}
\stackrel{\rm Pell}{=}
\frac{\Vol(S^{\rho-2})\cdot \sqrt{P^2k^2+1}}{\left(\sqrt{p_0^2k^2+1}\right)^\rho}\,,
\ee
and by using the fact that $P^2=p_0^2-p_1^2-p_2^2-p_3^2<p_0^2$, we get
\be
\Vol({\SD}(\pi_P))\leq \frac{\Vol(S^{\rho-2})}{\left(\sqrt{p_0^2k^2+1}\right)^{\rho-1}}
<\frac{\Vol(S^{\rho-2})}{(p_0k)^{\rho-1}}\leq \frac{\Vol(S^{\rho-2})}{p_0^{\rho-1}}\,.
\ee
\end{proof}

This result yields the following:
\begin{corollary}\label{cor:volume-finite-pell-bounds}
For any given positive number $\zeta>0$ there exist only finitely many Pell bounds $\pi_P$ such that $\Vol({\SD}(\pi_P))\geq \zeta$.
\end{corollary}
\begin{proof}
It follows from the previous Proposition that $P=\sum_{i=0}^{\rho-1} p_i B_i$ must satisfy
\be 
p_0\leq \sqrt[\rho-1]{\frac{\Vol(S^{\rho-2})}{\zeta}}\,.
\ee
Therefore, $P$ is contained in the compact set
\be 
W_\zeta:=\left[0,\sqrt[\rho-1]{\frac{\Vol(S^{\rho-2})}{\zeta}}\right]
\times 
\left[
-\sqrt[\rho-1]{\frac{\Vol(S^{\rho-2})}{\zeta}},
\sqrt[\rho-1]{\frac{\Vol(S^{\rho-2})}{\zeta}}
\right]^{\rho-2}\,.
\ee
Since the Néron-Severi group is a lattice in $\NS_\R(A)$, there only exist finitely many classes in $W_\zeta\cap \NS(A)$.

Using the same method as in Remark~\ref{rem:proper-lb-basechange} we can revert the base change and find a cube $[-m_{T,\zeta},m_{T,\zeta}]^{\rho}$ such that $P\in\NS(A)$ is contained in that cube.
\end{proof}

Recall that we constructed in Prop.~\ref{prop:seshadri-curve-submaximal-area} an open subcone on which every submaximal ample Seshadri curve is submaximal whenever we find an upper bound $R$ for the Seshadri constant with $\varepsilon(L)\leq R< \sqrt{L^2}$.
This subcone yields an open set in $\mathcal N(A)$ with positive volume $\zeta$ and its volume can be computed, for example, by using formulas for hyper-pyramids. 
By the previous results, there are only finitely many Pell bounds whose submaximality domain has volume $\zeta$.

Thus, it is left to show that we can find an upper bound $R$ for the Seshadri constant with $\varepsilon(L)\leq R< \sqrt{L^2}$.
If $\sqrt{L^2}\notin \Q$ then the Pell bound $\pi_L$ provides us with such an upper bound.
If $\sqrt{L^2}\in\Q$, then the proof of \cite[Thm.~3.14]{Bauer-Schmidt} shows that for an ample primitive symmetric line bundle $L$ we either have
\be 
\varepsilon(L)=\sqrt{L^2}\qquad \mbox{ or } \qquad \varepsilon(L)\leq \frac{2L^2-1}{2\sqrt{L^2}}<\sqrt{L^2}\,. 
%Can be improved to \frac{L^2-1}{\sqrt(L^2)}
\ee

Thus we have the necessary upper bound, and the construction given in Prop.~\ref{prop:seshadri-curve-submaximal-area} can be applied.
(Note that due to the construction of the upper bound $R$ via dimension-counting arguments of linear systems, this method is only applicable for rational ample line bundles.)

This will yield the second part of Thm.~\ref{introthm:seshadri-algorithm}, and completes the proof.

\begin{theorem}\label{thm:seshadri-ample}
Let $L$ be an ample line bundle and suppose that $L$ has an ample Seshadri curve which is submaximal for $L$.
Then there exists an effective algorithm to compute $\varepsilon(L)$ and the numerical data of all ample Seshadri curves of $L$.
\end{theorem}

\begin{proof}[Algorithm]
For a given basis $D_0,\dots, D_{\rho-1}$ for the Néron-Severi group we first compute the base change $T$ to the basis $B_0,\dots, B_{\rho-1}$ given by the Hodge index theorem.
The next step is to compute an upper bound of the Seshadri constant of $L$.
If $\sqrt{L^2}\notin\Z$ then the Pell bound of $L$ yields the upper bound $R:=\pi_L(L)$ and in the case $\sqrt{L^2}\in\Z$ we have the upper bound $R:=\frac{2L^2-1}{2\sqrt{L^2}}$ for the Seshadri constant.
Using the upper bound $R$, we compute a lower bound $\zeta$ for the volume of the domain on which every Seshadri curve is submaximal (see Prop.~\ref{prop:seshadri-curve-submaximal-area}).
It follows from Cor.~\ref{cor:volume-finite-pell-bounds} and Remark~\ref{rem:proper-lb-basechange} that we only have to consider Pell bounds $\pi_P$ with $P\in [-m_{T,\zeta},m_{T,\zeta}]^{\rho}$.
Since every ample Seshadri curve is represented by a Pell bound, the minimum over those Pell bounds computes the Seshadri constant, i.e., we have 
\be 
\varepsilon(L)=\min\set{\pi_P(L)\with \pi_P \mbox{ Pell bound with } P\in [-m_{T,\zeta},m_{T,\zeta}]^{\rho}}\,.
\ee

For the computation of the Seshadri curves, we use Thm.~\ref{thm:pell-bound-submaximal-curve-1-to-1} as follows:
Let $P$ be an ample primitive symmetric line bundle, such that $\varepsilon(L)$ coincides with $\pi_P(L)$.
Then, we apply the first part of the algorithm to compute $\varepsilon(P)$ and check whether there are elliptic curves or other Pell bounds which also compute $\varepsilon(P)$.
If no such elliptic curve or Pell bound exists, then the Pell divisor of $P$ yields an ample Seshadri curve of $L$ by Thm.~\ref{thm:pell-bound-submaximal-curve-1-to-1}~(iv).
\end{proof}

\begin{remark}\label{rem:seshadri-takes-upper-bound}
In the case of an ample line bundle $L$ with $\sqrt{L^2}\in\Z$, it is a priori unknown whether the Seshadri constant is strictly smaller or equal to $\sqrt{L^2}$.
To determine this, we first check whether or not there exists a submaximal elliptic curve by applying the algorithm used for Thm.~\ref{thm:elliptic-sesh}.
If there does not exist a submaximal elliptic curve, then we apply the algorithm used for Thm.~\ref{thm:seshadri-ample} under the additional assumption that the Seshadri constant is strictly smaller than $\sqrt{L^2}$.
Thus, we may assume that the upper bound $R:=\frac{2L^2-1}{2\sqrt{L^2}}$ holds and we can apply the algorithm.
If the algorithm yields
\be 
\varepsilon^*:=\min\set{\pi_P(L)\with \pi_P \mbox{ Pell bound with } P\in [-m_{T,\zeta},m_{T,\zeta}]^{\rho}} \leq \sqrt{L^2}\,,
\ee
then the Seshadri constant of $L$ satisfies $\varepsilon(L)=\varepsilon^*<\sqrt{L^2}$ since there exists a submaximal Pell divisor and, hence, a curve whose Seshadri constant is smaller than $\sqrt{L^2}$.
However, if the algorithm shows that $\varepsilon^* \geq \sqrt{L^2}$, then no such Pell bound exists and, thus, the assumption $\varepsilon(L)<\sqrt{L^2}$ was incorrect.
\end{remark}

\section{Structure of the Seshadri function on abelian surfaces}
\label{sec:plot}
The developed methods provide a tool to compute not only Seshadri constants of individual line bundles but also determine the numerical data of their submaximal Seshadri curves.
This gives access to the local structure of Seshadri functions since they are \textit{locally} piecewise linear whenever $\varepsilon(L)<\sqrt{L^2}$ holds where the linear segments are given by its Seshadri curves (see Prop.~\ref{prop:local-structur-seshadri-function-n-curves}).
In this section, we will further discuss the local and global structure of the Seshadri function.

\paragraph{Local Structure.}
To understand and describe the structure of the Seshadri function we will consider the set of rays
\be 
Y:=\set{L \subset \Nef(A)\setminus \set{0} \,\with\, \varepsilon \mbox{ is not piecewise linear in any neighborhood of } L}\,.
\ee
where the Seshadri function does not locally behave like a piecewise linear function.
By Prop.~\ref{prop:local-structur-seshadri-function-n-curves} the Seshadri function is locally piecewise linear if $\varepsilon(L)<\sqrt{L^2}$ holds.
The following result shows that the case $\varepsilon(L)<\sqrt{L^2}$ holds quite frequently and implies that $Y$ is a subset of a closed and nowhere dense set.
\begin{lemma}\label{lem:local-nowhere-dense}
Let $A$ be an abelian surface with Picard number $\rho\geq 2$.
Then the set 
\be 
\set{L\in \Nef(A)\,\with\, \varepsilon(L)=\sqrt{L^2}}\subset \Nef(A)
\ee
is closed and nowhere dense.
\end{lemma}
\begin{proof}
Let $W$ be the set in the statement.
Clearly, $W$ is closed since $\varepsilon$ and the upper bound $L\mapsto \sqrt{L^2}$ are continuous functions.
Now, assume to the contrary that there exists an open subset $U\subset \Nef(A)$ such that $W$ is dense in $U$.
We will derive a contradiction by constructing an open subset $V\subset U$ with $\varepsilon(L)<\sqrt{L^2}$ for all $L\in V$.
For this, we consider the open subcone $\mathcal C \subset \Nef(A)$ generated by $U$ and chose a basis $B_1,\dots, B_\rho$ of the Néron-Severi group with elements in $\mathcal C$.

We claim that there exists a line bundle $L\in \mathcal C$ with $\sqrt{L^2}\notin\Z$.
Let $S$ be the associated intersection matrix, then the existence of a line bundle $L$ with $\sqrt{L^2}\notin\Z$ is equivalent to the existence of a vector $(x_1,\dots,x_\rho)\in\N^{\rho}_{0}$ such that $x^TSx$ is not a perfect square.
We consider the polynomial function
\be 
F(X_1,\dots, X_\rho)=(X_1,\dots, X_\rho)S(X_1,\dots, X_\rho)^T\in\Z[X_1,\dots, X_\rho]\,.
\ee
Our assumption imply that $F(x)$ is a perfect square for every $x\in \N_0^\rho$ and, therefore, it follows from \cite[Thm.~2]{Murty:quadpoly} that there exists a polynomial function $G\in\Z[X_1,\dots,X_\rho]$ such that $F(X)=G^2(X)$.
This, however, implies that the Determinant of $S$ vanishes which is a contradiction to the Hodge index theorem.
Thus, we find a line bundle $L$ with $L\in \mathcal C$ with $\sqrt{L^2}\notin\Z$.

The submaximality cone $\SC(\pi_L)\subset \Nef(A)$ of the Pell bound $\pi_L$ yields an open subcone with $\varepsilon(M)<\sqrt{M^2}$ for all $M\in \SC(\pi_L)$.
Then, $V:=\SC(\pi_L)\cap U$ is a non-empty open subset of $U$ with $W\cap V=\emptyset$, which completes the proof.
\end{proof}

\paragraph{Global Structure.} 
For the global structure, we have to consider how the linear segments of the Seshadri function are put together globally.
One approach is to further study the structure of the set $Y$, where the Seshadri function is not piecewise linear.
We will first study the case where $Y$ is empty, i.e., where the Seshadri function is globally piecewise linear.
Contrary to what one might expect, the following result shows that Seshadri functions are not necessarily globally piecewise linear functions and, even more, piecewise linear Seshadri functions seem to be the exception.

\begin{proposition}\label{prop:piecewise-linear}
Let $A$ be an abelian surface such that the Seshadri function is piecewise linear.
Then $A$ is either an abelian surface with Picard number one or a non-simple abelian surface with Picard number two.
\end{proposition}
\begin{proof}
Assume that $A$ is a non-simple abelian surface with Picard number three or four.
Then $A$ has infinitely many elliptic curves passing through 0 and, thus, Prop.~\ref{prop:elliptic-seshadri-curve-on-cone} shows that $A$ has infinitely many submaximal Seshadri curves.
Therefore, the Seshadri function consists of infinitely many linear segments and, in particular, is not piecewise linear.

Now suppose that $A$ is a simple abelian surface with Picard number two or three.
We assume to the contrary that $A$ has a piecewise linear Seshadri function and, thus, the Seshadri function can be written as the minimum of finitely many curves, i.e.,
\be 
\varepsilon:\Nef(A)\to \R_{\geq 0}\,,\qquad L\mapsto \min\set{\frac{L\cdot C_i}{\mult_0 C_i}\,\with\, i=1,\dots,n}
\ee
where $C_1,\dots, C_n$ are finitely many ample curves.
Since the intersection of an ample line bundle with a non-trivial nef class is positive on an abelian surface (see Sect.~\ref{sec:sesh-elliptic-case}), it follows that $\varepsilon(L')>0$ for every $L'\in\partial\Nef(A)\setminus\set{0}$.
However, since $L'^2=0$ we have $\varepsilon(L')\leq \sqrt{L'^2}=0$ which is a contradiction.
\end{proof}
On non-simple abelian surfaces with Picard number two, the algorithm of Thm.~\ref{introthm:seshadri-algorithm} shows that the Seshadri function is piecewise linear in the following cases:
\begin{example}\label{ex:piecewise-linear}
Let $A$ be a non-simple abelian surface with Picard number two such that there exists a basis of the Néron-Severi group with one of the following intersection matrices
\be 
\left(
\begin{array}{cc}
	2 & 2 \\
	2 & 0 
	\end{array}
\right)\,,\qquad
\left(
\begin{array}{cc}
	2 & 3 \\
	3 & 0 
	\end{array}
\right)\,,\qquad 
\left(
\begin{array}{cc}
	0 & n \\
	n & 0 
	\end{array}
\right)\quad
\mbox{ for } n\in\set{1,\dots, 6}\,.
\ee
Then the Seshadri function of $A$ is piecewise linear.
The graph in Fig.~\ref{fig:seshadri-funktion-0-4-0} illustrates a piecewise linear Seshadri function with three linear segments.
\end{example}
It is unknown whether there are other cases where the Seshadri function is piecewise linear.
Computer-assisted computations even suggest that piecewise linear functions only appear in the cases above.
\begin{conjecture}\label{conj:finitely-piecewise}
Let $A$ be a non-simple abelian surface with Picard number two.
Then the Seshadri function of $A$ is piecewise linear precisely in the cases of Example~\ref{ex:piecewise-linear}.
\end{conjecture}
The conjecture suggests that most Seshadri functions on non-simple abelian surfaces with Picard number two are not piecewise linear.
The following example exhibits such a case.
\begin{example}\label{ex:non-piecewise-linear}
Let $A$ be a non-simple abelian surface with Picard number two such that the Néron-Severi group is generated by the classes of two elliptic curves $N_1, N_2$ with $N_1\cdot N_2=8$.
We claim that the Seshadri function is not piecewise linear.
To this end, we consider the line bundle $L:=\O_A(N_1+N_2)$ and the algorithm from Thm.~\ref{introthm:seshadri-algorithm} shows that $\varepsilon(L)=\sqrt{L^2}=4$.

We show that no linear segment $f$ of the Seshadri function satisfies $f(L)=\sqrt{L^2}=4$ and, thus, there are infinitely many linear segments accumulating near $L$.
We start with the linear segments generated by elliptic curves.
Since we have $L\cdot N_1=L\cdot N_2=8$, none of the elliptic curves compute the Seshadri constant of $L$.

Next, we consider the linear segments generated by submaximal ample curves.
We claim that no Pell bound $\pi_P$ and, hence, no linear segment from an ample curve satisfies $\pi_P(L)=\sqrt{L^2}$.
Let $P\equiv aN_1+bN_2$ with $a,b\in\N$ be an ample primitive line bundle with $\sqrt{P^2}=4\sqrt{ab}\notin\Z$, and denote by $(\ell,k)$ the primitive solution of the Pell equation $x^2-P^2y^2=1$.
Assume that the linear function $\pi_P$ satisfies $\pi_P(L)=\sqrt{L^2}=4$ which is equivalent to $2k(a+b)=\ell$.
By squaring both sides and using the Pell identity $\ell^2=1+16k^2ab$, we obtain
\be 
4k^2(a+b)^2=1+16k^2ab\,.
\ee
However, this yields $4k^2(a-b)^2=1$ which is impossible.
Therefore, we have $\pi_P(L)\neq \sqrt{L^2}$ for every Pell bound $\pi_P$.

This shows that in every open neighborhood of $L$ the Seshadri function consists of infinitely many linear segments.
Furthermore, one can show that the submaximality domains of the elliptic curves $N_1$, $N_2$ and the Pell bounds of $n^2N_1+(n+1)^2N_2$ and $(n+1)^2N_1+n^2N_2$ for $n\in\N$ cover $\Nef(A)\setminus\R_{>0} L$ and, thus, the Seshadri function is locally piecewise linear for all $M\in \Nef(A)$ with $M\notin \R_{>0} L=Y$.
This yields a function that is \engqq{almost} piecewise linear, which means that it is piecewise linear on every compact subset $K\subset \Nef(A)\setminus \R_{>0}L$.
The graph in Fig.~\ref{fig:seshadri-funktion-0-8-0} illustrates the Seshadri function of $A$ in the basis $B_0,B_1$ (see Sect.~\ref{sec:computing-seshc}).
\end{example}

\begin{remark}
A different perspective on the question of whether the Seshadri function is piecewise linear can be obtained by using the results of Cor.~\ref{introcor:isometries} and Cor.~\ref{cor:seshadricurves-isometry}.
These imply that the Seshadri function can never be piecewise linear whenever the subgroup $G\subset \Aut(\NS(A))$ of isometries that leaves the Seshadri function invariant is infinite.

For example, on \textit{simple} abelian surfaces with Picard two, the subgroup $G$ is isomorphic to either $\Z$ or $\Z\times \Z_2$.
This follows from the fact that the automorphism groups of binary quadratic forms are well understood (see \cite[Sect.~2.5]{Buchmann-Vollmer:BQF}).
Thus, the decomposition of the ample cone from Cor.~\ref{introcor:isometries} yields infinitely many subcones and, therefore, the Seshadri function is never piecewise linear since it has infinitely many linear segments.
On \textit{non-simple} abelian surfaces, on the other hand, \cite[Sect.~2.5]{Buchmann-Vollmer:BQF} shows that the subgroup $G$ is trivial or isomorphic to $\Z_2$.
Thus, this approach does not answer whether or not the Seshadri function on a non-simple abelian surface with Picard number two is piecewise linear and both cases do occur as we have seen in Example~\ref{ex:piecewise-linear} and Example~\ref{ex:non-piecewise-linear}.
\end{remark}

Naturally, one can ask what other \engqq{simple} structures are possible for the Seshadri function, apart from Example~\ref{ex:non-piecewise-linear}, where the set $Y$ consists of exactly one ray $\R_{>0}L$.
Another simple structure would emerge if it were to happen that the Seshadri function restricted to any subcone given by the decomposition from Cor.~\ref{introcor:isometries} is linear or at least piecewise linear.
By \cite{Bauer-Schmidt} this cannot occur on principally polarized abelian surfaces with real multiplication since on these surfaces the Seshadri function has infinitely many linear segments on each given subcone by Cor.~\ref{introcor:isometries}.
However, the following example shows that it does occur in Picard number three on self-products $E\times E$ of elliptic curves $E$ without complex multiplication:
\begin{example}\label{ex:product-elliptic}
Let $E$ be an elliptic curve without complex multiplication and consider the self product $A=E\times E$.
(For more details on the computations we refer to \cite{Bauer-Schulz} or \cite{Schmidt:integrality}.)
The Néron-Severi group is generated by the fibers $F_1=\{0\}\times E_2$, $F_2=E_1 \times \{0\}$ and the diagonal $\Delta$ form a basis of the Néron-Severi group $\NS(A)$ with the intersection matrix
\be
	\left(
	\begin{array}{ccc}
		0 & 1 & 1  \\
		1 & 0 & 1  \\
		1 & 1 & 0 
	\end{array}
	\right)
	\,
\ee
and it is known that every Seshadri constant is computed by an elliptic curve.
For an elliptic curve $N$ we denote by $\mathcal C(N)\subset \Nef(A)$ the closed subcone in which $N$ computes the Seshadri constants.

We claim that for every elliptic curve $N_1$ we find two more elliptic curves $N_2$ and $N_3$ such that $N_1,N_2, N_3$ form a basis of the Néron-Severi group with the same intersection matrix.
Thus, there exists an automorphism $\varphi:\NS(A)\to\NS(A)$ mapping $N_1,N_2, N_3$ to $F_1, F_2,\Delta$ and, therefore, $\varphi$ maps the subcone $\mathcal C(N_1)$ onto $\mathcal C(F_1)$.
This results in a decomposition of the ample cone into infinitely many subcones in the sense of Cor.~\ref{introcor:isometries} where the Seshadri function restricted to any such subcone is linear.

To show the existence of the elliptic curves $N_2$ and $N_3$ we use the representation from \cite[Lem.~2.4]{Schmidt:integrality}:
For any elliptic curve $N$ there exist coprime integers $\alpha,\beta\in\Z$, such that
\be 
N \equiv N_{\alpha,\beta} := (\beta^2-\alpha \beta)F_1 + (\alpha^2-\alpha \beta)F_2 + \alpha \beta\Delta\,
\ee
and, conversely, any coprime pair $\alpha,\beta\in\Z$ defines an elliptic curve this way.
Let $a,b\in\Z$ be a coprime pair such that $N_1\equiv N_{a,b}$.
Since $a,b$ are coprime, we find two coprime integers $c,d\in\Z$ with $ad-bc= 1$.
Then, the elliptic curves given by $N_2:=N_{c,d}$ and $N_3:=N_{a+c,b+d}$ form a basis as claimed.
\end{example}

\subsection{Structures and example plots in Picard number two}
\label{sec:structure-pic-2}
The previous section shows that there are abelian surfaces with Picard number two where the Seshadri function is piecewise linear, and by \cite{Bauer-Schmidt} there are cases where the Seshadri function is so-called broken linear.
But Example~\ref{ex:non-piecewise-linear} shows that there are Seshadri functions that are neither piecewise linear nor broken linear and, thus, there occur other structures for the Seshadri function.
Based on computational evidence, we will shed light on what other structures occur aside from piecewise linear and broken linear functions.
It turns out that we observe two kind of structures for the Seshadri function on abelian surfaces with Picard number two.
To give a precise definition, we will make use of the set
\be 
Z:=\set{t\in [-1,1]\,\with\, \varepsilon \mbox{ is not piecewise linear in } B_0+tB_1 }=Y\cap \mathcal N(A)\,
\ee 
where the complexity of $Z$ indicates the complexity of the Seshadri function.
For example, if the Seshadri function is piecewise linear, then $Z$ is empty.
If, on the other hand, the Seshadri function is broken linear, then $Z$ is a perfect set since the Seshadri function consists of infinitely many linear segments, which are never adjacent to each other.

In the first structure we observed, the set $Z$ consists of finitely many points and could only be observed on non-simple abelian surfaces with Picard number two.
In this case, the Seshadri function is piecewise linear on any compact interval that does not contain any elements of $Z$ and, thus, generalizes the case of piecewise linear Seshadri functions. 
In Fig.~\ref{fig:seshadri-funktion-0-4-0} and, respectively, Fig.~\ref{fig:seshadri-funktion-0-8-0} we provide examples with $Z=\emptyset$ and, respectively, $Z=\set{0}$.
We have seen in Example~\ref{ex:non-piecewise-linear} that the Seshadri function given in Fig.~\ref{fig:seshadri-funktion-0-8-0} is not piecewise linear in any neighborhood of $0$ and, thus, the graph consists of infinitely many linear segments accumulating at $0$.
This behavior is illustrated by a \engqq{gap} near $0$ that gets smaller as we add more Seshadri curves to the plot, but the gap will never close completely with finitely many curves. 

\begin{figure}[hbt!]
\centering
\begin{minipage}{0.45\textwidth}
\centering
\begin{tikzpicture}[x=3.3cm,y=2cm]
\draw[color=black] (0.8,1) node[above,draw] {$\matr{0 & 4 \\ 4  & 0}$};

\draw[->,color=black] (-1.05,0.) -- (1.05,0) node[right] {$t$};
\draw[->,color=black] (0,0) -- (0,1.2) node[above] {$\varepsilon(B_0+tB_1)$};
\foreach \x in {-1,-0.5,0,0.5,1}
\draw[shift={(\x,0)},color=black] (0pt,2pt) -- (0pt,-2pt) node[below] {\footnotesize $\x$};
\foreach \y in {0.25,0.75}
\draw[shift={(0,\y)},color=black] (2pt,0pt) -- (-2pt,0pt) node[left] {\footnotesize $\y$};
\foreach \x/\y in {-1/{-1/3},{-1/3}/{1/3},{1/3}/1} 
\draw(\x,{sqrt(1-\x*\x)})--(\y,{sqrt(1-\y*\y)});
\draw[dashed, domain=-1:1, samples=500] plot (\x, {sqrt(1-\x*\x)});
\end{tikzpicture}
\caption{\label{fig:seshadri-funktion-0-4-0}
The Seshadri function on an abelian surface, if it has the intersection matrix $\matr{0 & 4 \\ 4  & 0}$.}
\end{minipage}\hfill
\begin{minipage}{0.45\textwidth}
\centering
\begin{tikzpicture}[x=3.3cm,y=2cm]
\draw[color=black] (0.8,1) node[above,draw] {$\matr{0 & 8 \\ 8  & 0}$};

\draw[->,color=black] (-1.05,0.) -- (1.05,0) node[right] {$t$};
\draw[->,color=black] (0,0) -- (0,1.2) node[above] {$\varepsilon(B_0+tB_1)$};
\foreach \x in {-1,-0.5,0,0.5,1}
\draw[shift={(\x,0)},color=black] (0pt,1pt) -- (0pt,-1pt) node[below] {\footnotesize $\x$};
\foreach \y in {0.25,0.75}
\draw[shift={(0,\y)},color=black] (2pt,0pt) -- (-2pt,0pt) node[left] {\footnotesize $\y$};

\input{pointset_0-8-0}
\end{tikzpicture}
\caption{\label{fig:seshadri-funktion-0-8-0}
The Seshadri function on an abelian surface, if it has the intersection matrix $\matr{0 & 8 \\ 8  & 0}$.}
\end{minipage}
\end{figure}

The second case describes a function where $Z$ is a perfect set that generalizes the case of broken linear functions observed in \cite{Bauer-Schmidt}.
Computational data suggests that all Seshadri functions of simple abelian surfaces with Picard number two are of this type (we refer to \cite[Fig.~1-4]{Bauer-Schmidt} for example plots) and they also occur on non-simple surfaces:
In Fig.~\ref{fig:seshadri-funktion-0-9-0} and Fig.~\ref{fig:seshadri-funktion-4-5-0} we provide examples where computations indicate that the Seshadri functions belong to this type.
With the same argument as in Example~\ref{ex:non-piecewise-linear}, one can show that the Seshadri function in Fig.~\ref{fig:seshadri-funktion-0-9-0} is not piecewise linear at $\pm\tfrac{1}{3}\in Z$ and, thus, we have infinitely many linear segments approaching that point.
However, the underlying data suggests that the local structure near $\pm\tfrac{1}{3}$ seems to be much more intricate compared to Fig.~\ref{fig:seshadri-funktion-0-8-0} and the \engqq{gaps} only close very slowly if we add further Seshadri curves.
The data indicates the following structure:
\begin{itemize}\compact
\item[(i)]
Around every point of $[-1,1]\setminus Z$ there is an open interval, contained in $[-1,1]\setminus Z$, on which $\varepsilon$ is piecewise linear.
(This is clear by the definition of the set $Z$.)
\item[(ii)]
For every point $p\in Z$ there does not exist an interval $(a,p)$ or $(p,b)$, on which $\varepsilon$ is piecewise linear (i.e., on both sides there are infinitely many linear segments accumulating to $p$).
\item[(iii)]
If $I_1$ and $I_2$ are maximal open subintervals of $[-1,1]\setminus Z$, then $I_1$ and $I_2$ are not adjacent to each other (i.e., an endpoint of $I_1$ is never an endpoint of $I_2$).
\end{itemize}
Note that observation (ii) implies that a maximal open subinterval $I=(a,b)$ of $[-1,1]\setminus Z$ always contains infinitely many linear segments that accumulate to $a\neq -1$ and $b\neq 1$.
This property is in stark contrast to the behavior of broken linear functions since it requires the function to be linear on every maximal open subinterval of $[-1,1]\setminus Z$.
Moreover, the observation in (iii) implies that $Z$ is a perfect set.

In Fig.~\ref{fig:seshadri-funktion-4-5-0} we provide a graph for a non-symmetric Seshadri function where the data suggests that $Z$ satisfies (i) -- (iii) and, thus, has a similar structure as in the case of Fig.~\ref{fig:seshadri-funktion-0-9-0}.
Furthermore, we see two overlapping linear segments near $t=0.5$ which shows that there are line bundles with two submaximal curves.

\begin{figure}[hbt!]
\centering
\begin{minipage}{0.45\textwidth}
\centering
\begin{tikzpicture}[x=3.3cm,y=2cm]
\draw[color=black] (0.8,1) node[above,draw] {$\matr{0 & 9 \\ 9  & 0}$};

\draw[->,color=black] (-1.05,0.) -- (1.05,0) node[right] {$t$};
\draw[->,color=black] (0,0) -- (0,1.2) node[above] {$\varepsilon(B_0+tB_1)$};
\foreach \x in {-1,-0.5,0,0.5,1}
\draw[shift={(\x,0)},color=black] (0pt,2pt) -- (0pt,-2pt) node[below] {\footnotesize $\x$};

\foreach \y in {0.25,0.75}
\draw[shift={(0,\y)},color=black] (2pt,0pt) -- (-2pt,0pt) node[left] {\footnotesize $\y$};

\input{pointset_0-9-0}
\end{tikzpicture}
\caption{\label{fig:seshadri-funktion-0-9-0}
The Seshadri function on an abelian surface, if it has the intersection matrix $\matr{0 & 9 \\ 9  & 0}$.}
\end{minipage}\hfill
\begin{minipage}{0.45\textwidth}
\centering
\begin{tikzpicture}[x=3.3cm,y=2cm]
\draw[color=black] (0.8,1) node[above,draw] {$\matr{4 & 5 \\ 5  & 0}$};

\draw[->,color=black] (-1.05,0.) -- (1.05,0) node[right] {$t$};
\draw[->,color=black] (0,0) -- (0,1.2) node[above] {$\varepsilon(B_0+tB_1)$};
\foreach \x in {-1,-0.5,0,0.5,1}
\draw[shift={(\x,0)},color=black] (0pt,2pt) -- (0pt,-2pt) node[below] {\footnotesize $\x$};
\foreach \y in {0.25,0.75}
\draw[shift={(0,\y)},color=black] (2pt,0pt) -- (-2pt,0pt) node[left] {\footnotesize $\y$};

\input{pointset_4-5-0}
\end{tikzpicture}
\caption{\label{fig:seshadri-funktion-4-5-0}
The Seshadri function on an abelian surface, if it has the intersection matrix $\matr{4 & 5 \\ 5  & 0}$.}
\end{minipage}
\end{figure}

In all cases observed so far, computational data suggest that the set $Z$ is either perfect or finite.
It is conceivable that there are cases, where the set $Z$ is neither finite nor perfect, however, no such cases were observed.
Moreover, the underlying data even suggest that there are only finitely many cases where $Z$ is a finite set, which leads to the following generalization of the Conjecture~\ref{conj:finitely-piecewise}:

\begin{conjecture}
Let $A$ be an abelian surface with Picard number two.
\begin{itemize}\compact
\item[(i)]
If $A$ is simple then $Z$ is perfect.
\item[(ii)]
If $A$ is non-simple then $Z$ is either perfect or finite with $|Z|\in \set{0,1,2}$.
Moreover, there are only finitely many intersection matrices such that $Z$ is finite.
\end{itemize}
\end{conjecture}

%***************************************************************************** % Addresses

\footnotesize

   \bigskip
   Maximilian Schmidt
   Fachbereich Mathematik und Informatik,
   Philipps-Universit\"at Marburg,
   Hans-Meerwein-Stra\ss e,
   D-35032 Marburg, Germany.

   \nopagebreak
   \textit{E-mail address:} \texttt{schmid4d@mathematik.uni-marburg.de}

%*****************************************************************************

\end{document}